\documentclass[12pt]{article}
\author{Daniel R. Johnston\footnote{Supported by an Australian Mathematical Society Lift-off
Fellowship.}\\
Max Planck Institute for Mathematics, Bonn, Germany \\ johnston@mpim-bonn.mpg.de \and 
Tim Trudgian\footnote{Supported by Australian Research Council Discovery Project DP240100186.} \\
School of Science, UNSW Canberra, Australia \\
timothy.trudgian@unsw.edu.au}

\title{A round of Pintz to celebrate oscillations in sums}
\usepackage{indentfirst}
\usepackage{url}
\usepackage{enumerate}
\usepackage{amsthm}
\usepackage{amsmath,mathrsfs}
\usepackage{comment}
\usepackage{fullpage}
\usepackage{amssymb}
\usepackage{booktabs}

\newtheorem{thm}{Theorem}

\newtheorem{Lem}{Lemma}

\theoremstyle{definition}
\newtheorem*{remark}{Remark}

\usepackage[dvipsnames]{xcolor}
\usepackage{todonotes}
\begin{document}
\maketitle
\textit{Dedicated to J\'{a}nos Pintz on the occasion of his diamond jubilee}
\begin{abstract}
\noindent
    We explore a method, going back to Landau and developed by Pintz, for connecting sums of arithmetic functions with zero-free regions for $L$-functions. In particular, we make explicit a general result of Pintz of this form, showing how one can use arithmetical information to deduce information about zeroes of $L$-functions, rather than the other way around. As a prototype, we work through an example with the Riemann zeta-function and sums of the M\"obius function, but we also outline the utility of this method in general.
\end{abstract}
\section{Introduction and background}\label{leubald}
\noindent
The prime number theorem asserts that $\psi(x) \sim x$, where $\psi(x) = \sum_{n\leq x}\Lambda(n) =\sum_{p^{m}\leq x} \log p$.    One is often interested in the error of such approximations. To this end, define 
\begin{equation*}
    \Delta^{\psi}(x) := \psi(x) - x.
\end{equation*}
What can be said about the order of $\Delta^{\psi}(x)$? If there are no zeroes of $\zeta(s)$ to the right of some line
$\sigma = A<1$ then it is relatively easy to see that $\Delta^{\psi}(x) \ll x^{A + \epsilon}$. Hence, on the Riemann hypothesis (RH) we have $\Delta^{\psi}(x) \ll x^{1/2 + \epsilon}$. Up to the $\epsilon$ and implied constants this is best possible, since we also have $\Delta^{\psi}(x) = \Omega(\sqrt{x})$. 

This comes from following a standard recipe with these ingredients:
\begin{enumerate}
\item The coefficients in the Dirichlet series of $-\zeta'/\zeta(s)$ are $\Lambda(n)$.
\item The function $-\zeta'/\zeta(s)$ has poles at the zeroes of $\zeta(s)$.
\item There \textit{are} zeroes of $\zeta(s)$ with $\sigma = 1/2$.
\end{enumerate}
Landau extended this method (see, e.g., \cite{Pintz22}) to show that $\Delta^{\psi}(x)$ has oscillations depending on the real part of any zero of $\zeta(s)$. If we take $\rho_{0} = \beta_{0} + i\gamma_{0}$ as any complex zero of $\zeta(s)$, then it follows that $\Delta^{\psi}(x) = \Omega(x^{\beta_{0} - \epsilon}).$
Effective versions of Landau's result are obtainable via Tur\'{a}n's power sum method. While many authors have been involved in such estimates, we wish to focus on the framework developed\footnote{We take this moment to list some of Pintz's contributions that have since been made explicit for applications. These include work on the size of the error term in the prime number theorem \cite{PintzIng} (\cite{PTPNT,DJ4,Fiori,DJ3,Bellotti}), the mean-square of the error term \cite{PintzXY,Pintz85} (\cite{Brent}), weighted averages of the error term \cite{Pintz91} (\cite{DJ2}), and real zeroes for Dirichlet $L$-functions \cite{Pintz77} (\cite{Morrill}).} by Pintz \cite{Pintz80,Pintz80b,Pintz80c,Pintz82}. In \cite[Cor.\ 2.1]{Pintz22} it is shown that if $\rho_{0}$ is a zero with multiplicity $\nu$, then\footnote{In the ``annotated bibliography for comparative prime number
theory" (ABCPNT) \cite{ABCPNT}, the notation ${\mathfrak A}^\psi_{|1|}(Y)$ is suggested in place of $D^{\psi}(Y)$. However, here we have chosen to keep our notation more in line with that of Pintz's work.}
\begin{equation}\label{kitchen}
D^{\psi}(Y):= \frac{1}{Y} \int_{1}^{Y} |\Delta^{\psi}(x)|\, \mathrm{d}x \gg \frac{|\zeta^{(\nu)}(\rho_{0})|}{(\nu -1)! (1 + |\gamma_{0}|)^{5/2}}Y^{\beta_{0}}.
\end{equation}
Choosing a zero, for example, the lowest zero with positive ordinate, gives $\beta_{0} = 1/2$ and $\gamma_{0} = 14.134\ldots$. Hence (\ref{kitchen}) shows that  $D^{\psi}(Y)\gg Y^{1/2}$, and thus $\Delta^{\psi}(x) = \Omega(x^{1/2})$.

As shown in Pintz's work, this framework can be extended to other arithmetical functions appearing in analytic number theory. Notably, if $\mu(n)$ denotes the M\"obius function and ${M(x)=\sum_{n\leq x} \mu(n)}$, then the same approach gives that
\begin{equation}\label{kitchen2}
    D^{M}(Y)= \frac{1}{Y}\int_1^Y|M(x)|\mathrm{d}x\gg\frac{1}{(\nu -1)! (1 + |\gamma_{0}|)^{5/2}}Y^{\beta_{0}}(\log Y)^{\nu-1}.
\end{equation}

Compared to \eqref{kitchen}, the inequality \eqref{kitchen2} has the benefit of not depending on $\zeta^{(\nu)}(\rho_0)$. 

Now, once the implied constants in (\ref{kitchen2}) are known, one has an explicit lower bound on $D^{M}(Y)$. Suppose for example that $\nu=1$ so $D^{M}(Y) \geq c(\rho_{0}) Y^{\beta_{0}}$ for some calculable $c$. 
Suppose then that 
\begin{equation}\label{study}
|M(x)|\leq d\sqrt{x}, \quad (1\leq x \leq Y).
\end{equation}
It then follows trivially that $D^{M}(Y)Y^{-1/2}\leq 2d/3$. This then shows that
\begin{equation}\label{lounge}
Y^{\beta_0-1/2} \leq \frac{2d}{3c} \Rightarrow \beta_0 \leq 1/2+ \frac{\log (1/c) + \log (2d/3)}{\log Y}.
\end{equation}
Therefore, if we have verified (\ref{study}) for a large $Y$, we have a  zero-free region for simple zeroes of $\zeta(s)$. Almost certainly $c$ decreases with $\gamma_0$, so that the right side of (\ref{lounge}) increases with $\gamma_0$. Nevertheless, this zero-free region is non-trivial when $\beta_0<1$. It is `even more non-trivial' if it beats existing zero-free regions up to some finite height.

In this paper, we explore this concept in detail, giving an explicit version of \eqref{kitchen2} for simple zeroes of $\zeta(s)$. However, at the outset we concede that this approach, when applied to $\zeta(s)$, is doomed to fail: there is no chance that we will improve our knowledge of zero-free regions. This is due to the fact that the Riemann hypothesis has been verified to a large height \cite{PTRH} and also because a lot of work has been done on explicit versions of known zero-free regions (\cite{Kadiri,HoffTrudgian,MTY}). Nevertheless, with an explicit version of \eqref{kitchen2}, one can get an idea for how much information on $\zeta(s)$ would be obtainable if strong bounds on $M(x)$ were known. Moreover, one can also obtain explicit lower bounds for $D^M(Y)$, which seems to be closer to the original intention of Pintz's work.

We also express our results more generally, so that in future work one may also explore these ideas with more abstract $L$-functions. This is important, because even upon considering primes in arithmetic progressions, far less is known about the corresponding Dirichlet $L$-functions. In particular, the zero-free regions are not as sharp; our best numerical verifications of the Generalised Riemann hypothesis (GRH) are only to small heights when the conductors of the characters are large \cite{Platt,PTFuture}; finally, potential exceptional zeroes of real Dirichlet $L$-functions must be simple. For example, take $\chi$ mod $q$ for $q\approx 10^{20}$. We do not know that GRH is true for such characters up to height $t=1$. Then, as one considers more complicated $L$-functions, less and less is known about their zeroes. 

The outline of this paper is as follows. In Section \ref{hammerhead} we make explicit a general result from Pintz's work \cite{Pintz22} for simple zeroes, which includes (\ref{kitchen}) and \eqref{kitchen2} as specimens. In Section~\ref{nurse} we explore the explicit example \eqref{kitchen2} concerning $\zeta(s)$ and sums of the M\"{o}bius function $\mu(n)$. Finally, in Section~\ref{gws} we further discuss possible extensions of our work.

\section{Main theorem}\label{hammerhead}
In this section we prove our main result, which is an explicit version of the main theorem in \cite{Pintz22} for the case of simple zeroes ($\nu=1$). Our argument could certainly be extended to higher order zeroes, but we found the numerics to be poor. The case of simple zeroes is also by and large the most studied in the literature.

Our notation is essentially the same as Pintz's, but we use more parameters to give an explicit result. We have also modified the bound for $|F(s)|$ in \eqref{FGbounds} to give a better result for large imaginary parts. In what follows we will always express a complex number $s$ as $s=\sigma+it$ with $\sigma=\Re(s)$ and $t=\Im(s)$.

\begin{thm}[Explicit version of the Landau--Pintz result]\label{mainthm}
    Suppose that $A(x):\mathbb{R}\to\mathbb{C}$ is a complex valued function with $|A(x)|\leq c_Ax^C$ for some $c_A>0$ and $C> 0$, with
    \begin{equation*}
        \int_1^\infty\frac{A(x)}{x^{s+1}}\mathrm{d}x=\frac{F(s)}{G(s)},
    \end{equation*}
    where for some $\beta_0\in(0,1]$ and $0<c_0<\beta_0$, we have that $F(s)$ is analytic for $\sigma\geq\beta_0-c_0$, and $G(s)$ for $\sigma\geq -1$. Next, suppose that in these halfplanes, one has
    \begin{equation}\label{FGbounds}
        |F(s)|\leq c_F\max(1,|t|^{B_F})e^\sigma\quad |G(s)|\leq c_G\max(1,|t|^{B_G})e^{|\sigma|},
    \end{equation}
    for constants $c_F,B_F,c_G,B_G>0$. Then, if $G(s)$ has a simple zero $\rho_0=\beta_0+i\gamma_0$ and $F(\rho_0)\neq 0$, one has, for all $Y>e^{2C+4}$
    \begin{equation*}\label{maineq}
        \frac{1}{Y}\int_1^{Y}|A(x)|\mathrm{d}x\geq\frac{1}{e^{C+2}(1+|\gamma_0|)^{B_G} D_2}\left(\frac{|F(\rho_0)|\widetilde{Y}^{\beta_0}}{(\beta_0+2)^{B_G+2}}-\mathcal{E}(\widetilde{Y})\right)-\frac{c_AD_1(\log\widetilde{Y})e^{(C+1)(C+2)}}{(\log\widetilde{Y}-C-1)D_2},
    \end{equation*}
    where $\widetilde{Y}:=Ye^{-(C+2)}$,
    \begin{align}
        \mathcal{E}(y)&:=c_F\frac{e^{\beta_0-c_0}}{2\pi}e^{4/\log y}y^{\beta_0-c_0}\int_{-\infty}^{\infty}\frac{e^{-t^2/\log y}\max(1,|t|^{B_F})}{|c_0+it||2+\beta_0-c_0+it|^{B_G+2}}\mathrm{d}t,\label{mcEdef}\\
        D_1(y)&:=\frac{c_G}{\pi e}\left(\frac{1}{2}+\frac{1}{(y-2)(y+1)^{B_G+2}}\right),\label{D1def}\\
        D_2&:=c_G\frac{e}{\pi}\left(\frac{1-2^{-B_G-1/2}}{B_G+1}+\frac{1}{(1+\beta_0)}\right)\label{D2def}.
    \end{align}
\end{thm}
\begin{proof}
    We let $Y>e^{C+2}$, noting that we rescale $Y$ by an additional factor of $e^{C+2}$ at the end of the proof. Following \cite[p.\ 61]{Pintz22}, we set $\lambda=\log Y$ and introduce the auxiliary functions
    \begin{align}
        g(s)&:=\frac{G(s-1+i\gamma_0)}{(s-1-\beta_0)(s+1)^{B_G+2}},\notag\\
        r_{\lambda}(H)&:=\frac{1}{2\pi i}\int\limits_{(3)}e^{s^2/\lambda+Hs}g(s)\mathrm{d}s,\label{rdef}
    \end{align}
    where, as usual, $\int_{(3)}$ indicates that we are integrating over the vertical line $\Re(s)=3$ on the complex plane. We then let
    \begin{align}
        U:&=\frac{1}{Y}\int_1^\infty\frac{A(x)}{x^{i\gamma_0}}r_{\lambda}(\lambda-\log x)\mathrm{d}x\notag\\
        &=\frac{1}{2\pi iY}\int\limits_{(3)}e^{s^2/\lambda+\lambda s}g(s)\int_1^\infty\frac{A(x)}{x^{s+i\gamma_0}}\mathrm{d}x\mathrm{d}s\notag\\
        &=\frac{1}{2\pi iY}\int\limits_{(3)}e^{s^2/\lambda+\lambda s}\frac{F(s-1+i\gamma_0)}{(s-1-\beta_0)(s+1)^{B_G+2}}\mathrm{d}s.\label{Udef}
    \end{align}
By splitting the range of integration and bounding the integrals trivially we have
    \begin{equation}\label{mainuidentity}
        |U|\leq\sup_{H}|r_{\lambda}(H)|\frac{1}{Y}\int_1^{e^{C+2}Y}|A(x)|\mathrm{d}x+\left|\frac{1}{Y}\int_{e^{C+2}Y}^\infty\frac{A(x)r_{\lambda}(\lambda-\log x)}{x^{i\gamma_0}}\mathrm{d}x\right|.
    \end{equation}
For further optimisation, one could split the integration at a slightly different point, say $e^{C+3}Y$. However, we found that such an optimisation made little difference to our final results.

We now use the relation \eqref{mainuidentity} to extract a lower bound for the mean-value of $|A(x)|$. The main term will arise via a lower bound for
\begin{equation*}
    \frac{|U|}{\sup_H|r_\lambda(H)|},
\end{equation*}
and we will show that the second term on the right-hand side of \eqref{mainuidentity} is relatively small. We begin by finding a lower bound for $|U|$. Letting $\sigma_0=1+\beta_0-c_0$, we shift the line of integration in \eqref{Udef} from (3) to ($\sigma_0$). That is, we write
    \begin{equation*}
        \int_{(3)}=\oint_{C}-\int_{3+i\infty}^{\sigma_0+i\infty}-\int_{\sigma_0+i\infty}^{\sigma_0-i\infty}-\int_{\sigma_0-i\infty}^{3-i\infty}=I_1-I_2-I_3-I_4,\quad\text{say,}
    \end{equation*}
    where $C$ is the rectangle with corners $\sigma_0\pm\infty$ and $3\pm i\infty$, traversed anticlockwise.
    Now, to begin with, the residue theorem gives
    \begin{align}\label{firstcontour}
        \frac{1}{2\pi i Y}\oint_Ce^{s^2/\lambda+\lambda s}\frac{F(s-1+i\gamma_0)}{(s-1-\beta_0)(s+1)^{B_G+2}}\mathrm{d}s&=e^{(1+\beta_0)^2/\log Y}\frac{Y^{\beta_0}F(\rho_0)}{(\beta_0+2)^{B_G+2}}.
    \end{align}
    Next, by the bound \eqref{FGbounds} for $F(s)$, one sees that the horizontal integrals, $I_2$ and $I_4$ are zero. For example, writing
    \begin{equation*}
        I_2=\lim_{T\to\infty}\int_3^{\sigma_0}e^{(\sigma+iT)^2/\lambda+\lambda(\sigma+iT)}\frac{F(\sigma-1+iT+i\gamma_0)}{|\sigma-1-\beta_0+iT||\sigma+1+iT|^{B_G+2}}\mathrm{d}\sigma,
    \end{equation*}
    we have by \eqref{FGbounds} and the fact that $\sigma\in[\sigma_0,3]$,
    \begin{align*}
        |I_2|&\leq\lim_{T\to\infty}\int_{\sigma_0}^3\left|e^{(\sigma^2-T^2)/\lambda+\lambda\sigma}\right|\frac{|F(\sigma-1+iT+i\gamma_0)|}{|\sigma-1-\beta_0+iT||\sigma+1+iT|^{B_G+2}}\mathrm{d}\sigma\\
        &\leq\lim_{T\to\infty} e^{9/\log Y} Y^3 e^{-T^2/\log Y}c_Fe^3\frac{|T|^{B_F}}{|T|^{B_G+3}}\int_{\sigma_0}^3\mathrm{d}\sigma\\
        &=0
    \end{align*}
    as claimed. An identical argument then gives $|I_4|=0$. For the remaining integral $I_3$, we write $s=\sigma_0+it$, so that
    \begin{align}\label{I3first}
        \left|\frac{1}{2\pi iY}I_3\right|&\leq \frac{1}{2\pi Y}\int_{-\infty}^{\infty}e^{(\sigma_0^2-t^2)/\lambda+\lambda\sigma_0}\frac{|F(\beta_0-c_0+it+i\gamma_0)|}{|-c_0+it||2+\beta_0-c_0+it|^{B_G+2}}\mathrm{d}t\notag\\
        &\leq c_F\frac{e^{\beta_0-c_0}}{2\pi}e^{4/\log Y}Y^{\beta_0-c_0}\int_{-\infty}^{\infty}\frac{e^{-t^2/\lambda}\max(1,|t|^{B_F})}{|c_0+it||2+\beta_0-c_0+it|^{B_G+2}}\mathrm{d}t,
    \end{align}
    where we have used that $\sigma_0^2\leq 4$ and the bound \eqref{FGbounds} for $F(s)$. Note that for standard choices of parameters, the integral in \eqref{I3first} rapidly converges and is readily computable (numerically) so we leave it as is. Combining \eqref{firstcontour} and \eqref{I3first} we thus have
    \begin{align}\label{Ufinalbound}
        |U|\geq\frac{Y^{\beta_0}|F(\rho_0)|}{(\beta_0+2)^{B_G+2}}-\mathcal{E}(Y),
    \end{align}
    with $\mathcal{E}(\cdot)$ as in \eqref{mcEdef}.
    
Following \cite[p.\ 62]{Pintz22}, we now obtain an upper bound for the second term in \eqref{mainuidentity}. This is done via an upper bound for $r_{\lambda}(H)$ that is useful when $H$ is negative. In particular, if in \eqref{rdef} we move the line of integration of $r_{\lambda}(H)$ from the line $(3)$ to $(\lambda)$, we obtain
    \begin{equation}\label{rlamfirst}
        |r_{\lambda}(H)|\leq\frac{1}{2\pi}\int_{-\infty}^\infty e^{\lambda(H+1)-t^2/\lambda}\frac{|G(\lambda-1+it+i\gamma_0)|}{|\lambda-1-\beta_0+it||\lambda+1+it|^{B_G+2}}\mathrm{d}t.
    \end{equation}
    Here, we remark that there are no additional terms to consider when moving the line of integration via a rectangular contour. This is because the integrand of $r_{\lambda}(H)$ is analytic for $\Re(s)\geq -1$ and, as before when bounding $|U|$, the horizontal components of such a contour vanish as a result of the bound \eqref{FGbounds} on $G(s)$. We bound \eqref{rlamfirst} further as
    \begin{align}
        |r_{\lambda}(H)|&\leq\frac{c_G}{2\pi}\int_{-\infty}^\infty e^{\lambda(H+1)}\frac{\max(1,|t+\gamma_0|^{B_G})e^{\lambda-1}}{|\lambda-2+it||\lambda+1+it|^{B_G+2}}\mathrm{d}t\notag\\
        &\leq c_G\frac{e^{\lambda(H+2)}(1+|\gamma_0|)^{B_G}}{2\pi e}\int_{-\infty}^\infty\frac{\max(1,|t|^{B_G})}{|\lambda-2+it||\lambda+1+it|^{B_G+2}}\mathrm{d}t,\label{rlambda2eq}
    \end{align}
    where in the first inequality we used the bound \eqref{FGbounds} on $G(s)$ and in the second we used that
    \begin{equation}\label{tg0eq}
        \max(1,|t+\gamma_0|^{B_G})\leq \max(1,\left(|t|+|\gamma_0|\right)^{B_G})\leq (1+|\gamma_0|)^{B_G}\max(1,|t|^{B_G}).    
    \end{equation} 
    We bound the integral in \eqref{rlambda2eq} by noting
    \begin{equation*}
        \int_0^1\frac{1}{|\lambda-2+it||\lambda+1+it|^{B_G+2}}\mathrm{d}t\leq\int_0^1\frac{1}{(\lambda-2)(\lambda+1)^{B_G+2}}\mathrm{d}t=\frac{1}{(\lambda-2)(\lambda+1)^{B_G+2}},
    \end{equation*}
    and
    \begin{equation*}
        \int_1^\infty\frac{t^{B_G}}{|\lambda-2+it||\lambda+1+it|^{B_G+2}}\mathrm{d}t\leq\int_1^\infty\frac{1}{t^3}\mathrm{d}t=\frac{1}{2}
    \end{equation*}
    so that
    \begin{equation}\label{rlambdafinal}
        |r_{\lambda}(H)|\leq D_1(\lambda)\cdot (1+|\gamma_0|)^{B_G}e^{\lambda(H+2)},
    \end{equation}
    where $D_1(\cdot)$ is as in \eqref{D1def}. From \eqref{rlambdafinal} and the bound $|A(x)|\leq c_Ax^C$ we therefore deduce
    \begin{align*}
        \left|\frac{1}{Y}\int_{e^{C+2}Y}^\infty\frac{A(x)r_{\lambda}(\lambda-\log x)}{x^{i\gamma_0}}\mathrm{d}x\right|&\leq c_AD_1(\lambda)(1+|\gamma_0|)^{B_G}\int_{e^{C+2}Y}^\infty x^C e^{-\lambda(\log x-\lambda-1)}\mathrm{d}x\notag\\
        &=c_AD_1(\lambda)(1+|\gamma_0|)^{B_G}\int_{\lambda+C+2}^\infty e^{(C+1)y+\lambda+\lambda(\lambda-y)}\mathrm{d}y\notag\\
        &=c_AD_1(\lambda)(1+|\gamma_0|)^{B_G}e^{\lambda^2+\lambda}\int_{\lambda+C+2}^\infty e^{-y(\lambda-C-1)}\mathrm{d}y\notag\\
        &=c_AD_1(\lambda)\frac{e^{(C+1)(C+2)}}{\lambda-C-1}(1+|\gamma_0|)^{B_G}.
    \end{align*}
    Finally, we obtain an upper bound for $\sup_H |r_{\lambda}(H)|$ in \eqref{mainuidentity}. We do this similarly as with the previous bounds \eqref{rlamfirst}--\eqref{rlambdafinal} for $r_{\lambda}(H)$, but move the line of integration to $(0)$ to avoid any $H$ dependence. More precisely, if we move the line of integration in the definition \eqref{rdef} of $r_{\lambda}(H)$ from the line $(3)$ to $(0)$, then we obtain
    \begin{align}\label{r0first}
            |r_{\lambda}(H)|&\leq\frac{1}{2\pi}\int_{-\infty}^\infty e^{-t^2/\lambda}\frac{|G(-1+it+i\gamma_0)|}{|1+\beta_0+it||1+it|^{B_G+2}}\mathrm{d}t\notag\\
            &\leq c_G\frac{e}{2\pi}\int_{-\infty}^\infty\frac{\max(1,|t+\gamma_0|^{B_G})}{|1+\beta_0+it||1+it|^{B_G+2}}\mathrm{d}t\notag\\
            &\leq c_G\frac{e(1+|\gamma_0|)^{B_G}}{2\pi}\int_{-\infty}^\infty\frac{\max(1,|t|^{B_G})}{|1+\beta_0+it||1+it|^{B_G+2}}\mathrm{d}t.
    \end{align}
    Then, similar to before, we bound the integral in \eqref{r0first} by noting that
    \begin{equation*}
        \int_0^1\frac{1}{|1+\beta_0+it||1+it|^{B_G+2}}\mathrm{d}t\leq\int_0^1\frac{1}{1+\beta_0}\mathrm{d}t=\frac{1}{(1+\beta_0)},
    \end{equation*}
    and
    \begin{equation*}
        \int_1^\infty\frac{t^{B_G}}{|1+\beta_0+it||1+it|^{B_G+2}}\mathrm{d}t\leq\int_1^\infty\frac{t^{B_G}}{(\sqrt{1+t^2})^{B_G+3}}\mathrm{d}t=\frac{1-2^{-B_G-1/2}}{B_G+1}
    \end{equation*}
    so that
    \begin{equation}\label{r0final}
        |r_{\lambda}(H)|\leq (1+|\gamma_0|)^{B_G}D_2,
    \end{equation}
    where $D_2$ is as in \eqref{D2def}. Substituting \eqref{Ufinalbound}, \eqref{rlambdafinal} and \eqref{r0final} back into \eqref{mainuidentity} then gives
    \begin{align*}
        \frac{1}{Y}\int_1^{e^{C+2}Y}|A(x)|\mathrm{d}x&\geq\frac{1}{\sup_H|r_{\lambda}(H)|}\left(|U|-\left|\frac{1}{Y}\int_{e^{C+2}Y}^\infty\frac{A(x)r_{\lambda}(\lambda-\log x)}{x^{i\gamma_0}}\mathrm{d}x\right|\right)\\
        &\geq\frac{1}{(1+|\gamma_0|)^{B_G}D_2}\left(\frac{Y^{\beta_0}|F(\rho_0)|}{(\beta_0+2)^{B_G+2}}-\mathcal{E}(Y)\right)-\frac{c_AD_1(\lambda)e^{(C+1)(C+2)}}{(\lambda-C-1)D_2}.
    \end{align*}
 Finally, performing the rescaling $Y\to Ye^{-(C+2)}=\widetilde{Y}$, or equivalently $\lambda\to\lambda-C-2$, then yields the desired result.
\end{proof}

\section{An explicit example using $M(x)$ and $\zeta(s)$}\label{nurse}
As discussed in the introduction, we shall connect the zeroes of $\zeta(s)$ to the mean value of the Mertens' function $M(x)= \sum_{n\leq x} \mu(n)$. Consider
\begin{equation}\label{MertDirichlet}
    \int_1^\infty\frac{M(x)}{x^{s+1}}\mathrm{d}x=\frac{s-1}{s(s-1)\zeta(s)}
\end{equation}
and take
\begin{equation}\label{MertFG}
    F(s)=s-1\quad\text{and}\quad G(s)=s(s-1)\zeta(s),
\end{equation}
which are both analytic functions.

\subsection{Bounds for $F(s)$ and $G(s)$}
To be able to apply Theorem \ref{mainthm} for $A(x)=M(x)$ via \eqref{MertDirichlet}, we require bounds of the form \eqref{FGbounds} for $F$ and $G$ as defined in \eqref{MertFG}. To do so, we first give a simple lemma.
\begin{Lem}\label{esiglem}
    Let $s=\sigma+it$. Then,
    \begin{equation}\label{sigmaid1}
        |\sigma+it|\leq\sqrt{2}\max(1,|\sigma|) \max(1,|t|),
    \end{equation}
    and for any $r\geq 0$,
    \begin{equation}\label{signaid2}
        \max(1,|\sigma|^r)\leq\max\left(1,\left(\frac{r}{e}\right)^r\right)e^{|\sigma|}.
    \end{equation}
\end{Lem}
\begin{proof}
    We begin with proving \eqref{sigmaid1}. Here, for $|t|\leq 1$, one has
    \begin{equation*}
        |\sigma+it|=\sqrt{\sigma^2+t^2}\leq \sqrt{\sigma^2+1}
    \end{equation*}
    and if $|t|>1$
    \begin{equation*}
        |\sigma+it|=\sqrt{\sigma^2+t^2}\leq|t|\sqrt{\sigma^2+1}
    \end{equation*}
    so that
    \begin{equation*}
        |\sigma+it|\leq\sqrt{\sigma^2+1}\:\max(1,|t|).
    \end{equation*}
    Now, when $|\sigma|\leq 1$, we have
    \begin{equation*}
        \sqrt{\sigma^2+1}\leq \sqrt{2}
    \end{equation*}
    and when $|\sigma|>1$,
    \begin{equation*}
        \sqrt{\sigma^2+1}\leq\sqrt{2}|\sigma|
    \end{equation*}
    so that
    \begin{equation*}
        \sqrt{\sigma^2+1}\leq\sqrt{2}\max(1,|\sigma|)
    \end{equation*}
    and \eqref{sigmaid1} follows. We now consider \eqref{signaid2}. Here, standard calculus arguments show that 
    \begin{equation*}
        \frac{|\sigma|^r}{e^{|\sigma|}}
    \end{equation*}
    is maximised at $|\sigma|=r$, and the desired result follows.
\end{proof}
From here, we have the following bound for $|F(s)|$. 
\begin{Lem}\label{Fzetalem}
    Let $F(s)=s-1$ with $\sigma\geq 0$. Then,
    \begin{equation*}
        |F(s)|\leq \sqrt{2}\max(1,|t|)e^{\sigma}.
    \end{equation*}
\end{Lem}
\begin{proof}
    Follows immediately from Lemma \ref{esiglem}.
\end{proof}
We now give a bound for $|G(s)|$. Understandably, this requires more work and we will need to use the Phragm\'en--Lindel\"of theorem to get a result for all $\sigma\geq -1$.
\begin{Lem}\label{Gzetalem}
    Let $G(s)=s(s-1)\zeta(s)$. Then, for all $\sigma\geq -1$, one has
    \begin{equation*}
        |G(s)|\leq 13.38\max(1,|t|^{7/2})e^{|\sigma|}.
    \end{equation*}
\end{Lem}
\begin{proof}
    We begin by bounding $G(s)$ on the half-line. We combine the bounds
    \begin{align*}
        \left|\zeta\left(\frac{1}{2}+it\right)\right|&\leq 1.461,\qquad (0\leq t\leq 3)\\
        \left|\zeta\left(\frac{1}{2}+it\right)\right|&\leq 0.595 t^{1/6}\log t,\qquad (3\leq t<200)\\
        \left|\zeta\left(\frac{1}{2}+it\right)\right|&\leq \frac{4}{(2\pi)^{1/4}}t^{1/4}, \qquad (t\geq 200)
    \end{align*}
    given in \cite{Hiary2016,HPY2024} to deduce that
    \begin{equation*}
        \left|\zeta\left(\frac{1}{2}+it\right)\right|\leq 2.53t^{1/4}\leq 2.53\left|\frac{1}{2}+it\right|^{1/4}
    \end{equation*}
    for all $t\geq 0$, and by extension
    \begin{equation}\label{g12tbound}
        \left|G\left(\frac{1}{2}+it\right)\right|\leq2.53\left|\frac{1}{2}+it\right|^{9/4}.
    \end{equation}
    With a view to applying a convexity estimate, we next bound the $G(s)$ on the $\sigma_0$-line, for some arbitrary $\sigma_0\geq 2$. Here, one simply has
    \begin{equation*}
        |G(\sigma_0+it)|\leq\zeta(\sigma_0)\left|\sigma_0+it\right|^2,
    \end{equation*}
    noting that $|\zeta(\sigma_0+it)|\leq \zeta(\sigma_0)$ whenever $\sigma_{0}>1$. Using the Phragm\'en--Lindel\"of theorem (see \cite[Theorem 22]{Rad1959}) and \eqref{sigmaid1} of Lemma \ref{esiglem} one thus has
    \begin{align}\label{Phrageq1}
        |G(\sigma+it)|&\leq 2.53^{\frac{\sigma_0-\sigma}{\sigma_0-1/2}}|\sigma+it|^{\frac{9}{4}\frac{\sigma_0-\sigma}{\sigma_0-1/2}}\zeta(\sigma_0)^{\frac{\sigma-1/2}{\sigma_0-1/2}}|\sigma+it|^{\frac{2\sigma-1}{\sigma_0-1/2}}\notag\\
        &\leq 2.53^{\frac{\sigma_0-\sigma}{\sigma_0-1/2}}\zeta(\sigma_0)^{\frac{\sigma-1/2}{\sigma_0-1/2}}\left(\sqrt{2}\max(1,|\sigma|) \max(1,|t|)\right)^{\frac{1}{\sigma_0-1/2}\left(\frac{9}{4}(\sigma_0-\sigma)+2\sigma-1\right)}
    \end{align}
    for all $1/2\leq\sigma\leq \sigma_0$. Since $\sigma_0\geq 2$, we have $\zeta(\sigma_0)\leq\zeta(2)=\pi^2/6$ and therefore
    \begin{align}\label{tara}
        2.53^{\frac{\sigma_0-\sigma}{\sigma_0-1/2}}\zeta(\sigma_0)^{\frac{\sigma-1/2}{\sigma_0-1/2}}2^{\frac{0.5}{\sigma_0-1/2}\left(\frac{9}{4}(\sigma_0-\sigma)+2\sigma-1\right)}&\leq 2.53\cdot 2^{9/8}\\
        &\leq 5.52,\notag
    \end{align}
    with the upper bound in \eqref{tara} attained $\sigma=1/2$. For the other factor in \eqref{Phrageq1}, the maximum is also attained at $\sigma=1/2$ with
    \begin{align}\label{sigiteq1}
        \left(\max(1,|\sigma|) \max(1,|t|)\right)^{\frac{1}{\sigma_0-1/2}\left(\frac{9}{4}(\sigma_0-\sigma)+2\sigma-1\right)}&\leq\max(1,|\sigma|^{9/4}|)\max(1,|t|^{9/4})\notag\\
        &\leq\max(1,|t|^{9/4})e^{|\sigma|},
    \end{align}  
    where in the last step we applied \eqref{signaid2} of Lemma \ref{esiglem}. Inserting (\ref{tara}) and (\ref{sigiteq1}) into \eqref{Phrageq1} gives
    \begin{equation}\label{G0bound}
        |G(\sigma+it)|\leq 5.52\max(1,|t|^{9/4})e^{|\sigma|}.
    \end{equation}
    Since $\sigma_0\geq 2$ was arbitrary, this bound holds for all $\sigma\geq 1/2$. It thus remains to consider the case where $-1\leq\sigma\leq 1/2$. For this case, we use the functional equation (\cite[Chapter 2.1]{Titchmarsh})
    \begin{equation*}
        \xi(s)=\xi(1-s),
    \end{equation*}
    where
    \begin{equation*}
        \xi(s)=\frac{1}{2}s(s-1)\pi^{-\frac{1}{2}s}\:\Gamma\left(\frac{1}{2}s\right)\zeta(s)=\frac{1}{2}\pi^{-\frac{1}{2}s}\:\Gamma\left(\frac{1}{2}s\right)G(s).
    \end{equation*}
    This implies that
    \begin{equation*}
        G(s)=\frac{\pi^s}{\sqrt{\pi}}\frac{\Gamma\left(\frac{1}{2}-\frac{1}{2}s\right)}{\Gamma\left(\frac{1}{2}s\right)}G(1-s)
    \end{equation*}
    and hence
    \begin{align*}
        |G(-1+it)|&\leq\frac{1}{\pi\sqrt{\pi}}\left|\frac{\Gamma\left(1-\frac{it}{2}\right)}{\Gamma\left(-\frac{1}{2}+\frac{it}{2}\right)}\right|G(2)\\
        &\leq\frac{1}{\pi\sqrt{\pi}}\left|\frac{\Gamma\left(1-\frac{it}{2}\right)}{\Gamma\left(-\frac{1}{2}+\frac{it}{2}\right)}\right||1+it||2+it|\:\zeta(2).
    \end{align*}
    To bound the ratio of Gamma functions, we use the general identity
    \begin{equation}\label{Gammabound}
        \left|\frac{\Gamma\left(\frac{1-s}{2}\right)}{\Gamma\left(\frac{s}{2}\right)}\right|\leq\frac{1}{2^{1/2-\sigma}}|s-[\sigma]+1|^{1/2+[\sigma]-\sigma}\prod_{j=1}^{-[\sigma]}|s+j-1|,\qquad(\sigma<0),
    \end{equation}
    proved in  \cite[pp.\ 1469--1470]{Bennettetal}, where $[\sigma]$ denotes the closest integer to $\sigma$. From \eqref{Gammabound},
    \begin{equation*}
        \left|\frac{\Gamma\left(1-\frac{it}{2}\right)}{\Gamma\left(-\frac{1}{2}+\frac{it}{2}\right)}\right|\leq\frac{1}{2^{3/2}}|1+it|^{3/2}
    \end{equation*}
    and thus
    \begin{align}\label{gm1tbound}
        |G(-1+it)|\leq\frac{1}{\pi\sqrt{\pi}}|-1+it|^{5/2}|2+it|\:\zeta(2)&\leq\frac{1}{\pi\sqrt{\pi}}\cdot\frac{6}{\pi^2}\cdot 2\:|-1+it|^{7/2}\notag\\
        &\leq 0.6\:|-1+it|^{7/2}.
    \end{align}
    As before, we can now apply the Phragm\'en--Lindel\"of theorem with \eqref{g12tbound} and \eqref{gm1tbound} to obtain
    \begin{align*}
        |G(\sigma+it)|&\leq 0.6^{\frac{2}{3}(\frac{1}{2}-\sigma)}|\sigma+it|^{\frac{7}{3}(\frac{1}{2}-\sigma)}2.53^{\frac{2}{3}(\sigma+1)}|\sigma+it|^{\frac{3}{2}(\sigma+1)}\\
        &\leq 0.6^{\frac{2}{3}(\frac{1}{2}-\sigma)}2.53^{\frac{2}{3}(\sigma+1)}\left(\sqrt{2}\max(1,|\sigma|) \max(1,|t|)\right)^{\frac{7}{3}(\frac{1}{2}-\sigma)+\frac{3}{2}(\sigma+1)}
    \end{align*}
    for $-1\leq \sigma\leq 1/2$. Hence
    \begin{align}\label{tara2}
        |G(\sigma+it)|\leq 0.6^{\frac{2}{3}(\frac{1}{2}-\sigma)}2.53^{\frac{2}{3}(\sigma+1)}(\sqrt{2})^{\frac{7}{3}(\frac{1}{2}-\sigma)+\frac{3}{2}(\sigma+1)}&\leq 2.53(\sqrt{2})^{9/4}\leq 5.52,
    \end{align}
    with the upper bound in \eqref{tara2} attained at $\sigma=1/2$. Then, by Lemma 1,
    \begin{equation*}
        \left(\max(1,|\sigma|) \max(1,|t|)\right)^{\frac{7}{3}(\frac{1}{2}-\sigma)+\frac{3}{2}(\sigma+1)}\leq\left(\frac{7}{2e}\right)^{7/2}\max(1,|t|^{7/2})e^{|\sigma|}.
    \end{equation*}
    Therefore,
    \begin{equation}\label{gs112bound}
        |G(\sigma+it)|\leq 5.52\left(\frac{7}{2e}\right)^{7/2}\max(1,|t|^{7/2})e^{|\sigma|}\leq 13.38\max(1,|t|^{7/2})e^{|\sigma|}
    \end{equation}
    for $-1\leq \sigma\leq 1/2$. By \eqref{G0bound}, we  have that \eqref{gs112bound} also holds for $\sigma\geq 1/2$, proving the lemma.
\end{proof}

\subsection{Example applications with $M(x)$}

Assuming the Riemann hypothesis, one has $M(x)=O(x^{1/2+\epsilon})$ for any $\epsilon>0$. A well-studied conjecture of Mertens is that $|M(x)|\leq \sqrt{x}$ for all $x$. This is known to be false; Pintz \cite{Mertens} gave an explicit bound on the first counterexample. Nevertheless, for all computed values of $x$ we do indeed have $|M(x)|\leq \sqrt{x}$. Indeed, Hurst \cite{Hurst} showed that 
\begin{equation}\label{Hursteq}
    |M(x)|\leq 0.571 \sqrt{x}, \qquad (2\leq x \leq 10^{16}).
\end{equation}
Making use of Lemmas \ref{Fzetalem} and \ref{Gzetalem}, we can take $c_A=1$, $C=1$, $c_F=\sqrt{2}$, $B_F=1$, $c_G=13.38$ and $B_G=7/2$ in Theorem \ref{mainthm} for $A(x)=M(x)$. Currently, it is known that the first counterexample to Mertens' conjecture occurs for some $x\leq \exp(1.96\cdot 10^{19})$, see \cite{KimNguyen}. Suppose that this bound was near optimal and we had $|M(x)|\leq \sqrt{x}$ for all $x\leq Y = \exp(10^{19})$. In this case, we would have that
\begin{equation}\label{Mertbound1}
    \frac{1}{Y}\int_1^Y|M(x)|\mathrm{d}x\leq\frac{2}{3}\sqrt{Y}=(0.666\ldots)\times\exp\left(\frac{10^{19}}{2}\right).
\end{equation}
Using Theorem \ref{mainthm} with $c_0=0.1$, \eqref{Mertbound1} implies that there must be no simple zeroes of $\zeta(s)$ with real part greater than $\beta_0=0.51$ and imaginary part less than $\gamma_0=\exp(10^{16})$! This is a truly large amount of non-trivial information about the zeroes of $\zeta(s)$, and is indicative of the current futility of trying to precisely find the first counterexample to Mertens' conjecture. 

As a less exaggerated example, say we only knew that Mertens' conjecture held for $x\leq Y=10^{80}$. Then, by the same argument, Theorem \ref{mainthm} yields that there are no simple zeroes of $\zeta(s)$ with real part greater than $\beta_0=0.99$ and imaginary part less than $\gamma_0=10^{13}$. If we may be so bold: this is not completely terrible! In particular, the Riemann hypothesis is only known to hold for imaginary height $|\gamma|\leq 3\cdot 10^{12}$ \cite{PTRH}, and for larger heights less information is known via zero-free regions \cite{MTY}. However, the main problem is that such a large-scale computation on $M(x)$ seems infeasible. If we instead use \eqref{Hursteq}, we get no simple zeroes with real part greater than $\beta_0=0.99$ and imaginary part less than $\gamma_0=5$. This \textit{really is} terrible, and suggests that at present this method is not useful for bounding the real part of simple zeroes with large imaginary part --- cf.\ our remarks in \S \ref{leubald}.

Looking at asymptotics, we find that this approach (with $B_G=7/2$) is able to give non-trivial information about the zeroes $\rho=\beta+i\gamma$ of $\zeta(s)$ for approximately
    $|\gamma|\lesssim Y^{1/5}$,
which is roughly illustrated in our above examples for large $Y$. This is notably worse than if one goes in the other direction. From \cite{DJ} one has that verifying the Riemann hypothesis up to
imaginary height $T$ gives good bounds on prime counting functions up to height $\approx T^2$. Thus, this reverse method is essentially worse by a power of 10. However, it is often easier to compute bounds on arithmetic functions instead of computing zeroes of $L$-functions. Thus, this method may still have teeth when one is interested in low-lying zeroes of more complicated $L$-functions. We also wish to emphasise that our approach is essentially the first of its kind, and so one would hope that further improvements may be developed in future work that are more computationally reasonable.

As a final example, we remark that our Theorem \ref{mainthm} gives a general way of producing explicit lower bounds for the mean values of an arithmetic function. For example, continuing on the above example with $M(x)$, we find that 
\begin{equation}\label{Mertenmeanroot}
    \frac{1}{Y}\int_1^Y|M(x)|\mathrm{d}x\geq (7\cdot 10^{-9})\sqrt{Y}, \quad Y\geq 10^{20},
\end{equation}
by applying Theorem \ref{mainthm} with $\beta_0=1/2$, $c_0=0.1$ and $\gamma_0=14.1347\ldots$ (the location of the first non-trivial zero of $\zeta(s)$). We remark that the inequality \eqref{Mertenmeanroot} is similar to, albeit slightly weaker, than that in Pintz's paper \cite{Pintz82}.

\begin{remark}
    Since the time of writing, S.~Kim has informed us of a result from Pintz's 1987 paper~\cite[Theorem~A]{Mertens}, see also~\cite[Theorem~C]{saout13}. Here, in the context of Mertens' function $M(x)$, it is shown that our bound in Theorem~\ref{mainthm} can be replaced with 
    \begin{equation}\label{saouteq}
        \frac{1}{Y}\int_0^Y|M(x)|\mathrm{d}x>\frac{Y^{\beta_0}}{|\gamma_0|^5},\qquad Y\geq|\gamma_0|^5
    \end{equation}
    for a zeta zero $\beta_0+i\gamma_0$ off the half-line. That is, in the case where $\beta_0>1/2$. The inequality \eqref{saouteq} is not as sharp as ours when $|\gamma_0|$ is large, but has the advantage of holding for all zeros of $\zeta(s)$ off the half-line, with no assumption on simplicity.
\end{remark}

\section{Conclusion and extensions}\label{gws}
In this paper, we have given an explicit version of a result due to Pintz~\cite{Pintz22}, linking mean-values of arithmetical functions with information on the zeroes of $L$-functions. In Section \ref{nurse}, the application of this approach to $M(x)$ and $\zeta(s)$ is thoroughly explored. While we could have detailed explicitly our analysis for other $L$-functions here, we refrained from doing so as this would substantially increase the length of this paper. In particular, to apply our technique to the zeroes of a function $L(s)$, one typically requires explicit bounds for $|L(1/2+it)|$. However, it is rare to find such bounds in the literature, whereas such bounds for $\zeta(s)$ are abundant, including those from \cite{Hiary2016,HPY2024} which we used in our analysis. Similarly, we are not aware of any existing large-scale computations for arithmetic functions such as
\begin{equation*}
    M(x,\chi)=\sum_{n\leq x}\chi(n)\mu(n),
\end{equation*}
which is the analogue of $M(x)$ for a Dirichlet character $\chi$.

We briefly mention some applications to other simple arithmetical functions, which could have easily been used in place of our choice of $M(x)$ in Section \ref{nurse}. The result in \cite[(2.7)]{Pintz22} looks at partial sums of the Liouville function $\lambda(n)$, numerical computations for which have been made in \cite{MosTru,MT2017,Dickens,MTIJNT}.
The result in \cite[(2.9)]{Pintz22} concerns the distribution of $k$-free numbers. This builds on Pintz's earlier work \cite{Pintz83}; numerical computations exist in \cite{MOST}. 

We conclude by noting that, throughout this article, we have used pointwise bounds such as (\ref{study}) in integrals such as (\ref{kitchen}). One should really proceed directly and aim at verifying that the integral is small. For example, in \cite[p.\ 1021]{Hurst} it is shown that $|M(x)|$ is often substantially less than the global bound obtained of $0.571x^{1/2}$. Computing the integral of the error term would lead to a small saving, which would potentially be useful in more general settings.

\section*{Acknowledgements}
We are grateful to Roger Heath-Brown, Seungki Kim, Greg Martin and the reviewer for their comments. DRJ is grateful to the Max Planck Institute for Mathematics in Bonn for its hospitality and financial support.

 \end{document}